\documentclass{amsart}
\usepackage{amsfonts,color}
\usepackage{amssymb}

\newtheorem{theorem}{Theorem}

\newtheorem{lemma}[theorem]{Lemma}

\newtheorem{proposition}[theorem]{Proposition}

\numberwithin{equation}{section}

\begin{document}
\title{Improved Beckner-Sobolev inequalities on K\"ahler manifolds}
\author{Fabrice Baudoin}
\email{fabrice.baudoin@uconn.edu}
\address{Department of Mathematics, University of Connecticut, Storrs, CT
06268, USA}
\author{Ovidiu Munteanu}
\email{ovidiu.munteanu@uconn.edu}
\address{Department of Mathematics, University of Connecticut, Storrs, CT
06268, USA}

\begin{abstract}
We prove new Beckner-Sobolev type inequalities on compact K\"{a}hler
manifolds with positive Ricci curvature. As an application, we obtain a
diameter upper bound that improves the Bonnet-Myers bound.
\end{abstract}

\thanks{The first author was partially supported by NSF grant DMS-1660031 and the Simons Foundation.
The second author was partially supported by NSF grant DMS-1506220.}
\maketitle

\section{Introduction}

Sobolev inequalities have been long studied on Riemannian manifolds. It is
well known that on a compact Riemannian manifold $\left( M^{n},g\right) $ we
have a family of inequalities of the type 
\begin{equation*}
\left( \int_{M}\left\vert \phi \right\vert ^{p}\right) ^{\frac{2}{p}}\leq
A\int_{M}\left\vert \nabla \phi \right\vert ^{2}+B\int_{M}\phi ^{2},
\end{equation*}%
for any $\phi \in C^{\infty }\left( M\right) $ and $\frac{2n}{n-2}\geq p>2$,
where $A$ and $B$ are some fixed constants that may depend on $M$.

It is an important question to obtain sharp bounds for the constants $A$ or $%
B$. An important result of Hebey-Vaugon \cite{HV} says that for $p=$ $\frac{%
2n}{n-2}$ the constant $A$ may be taken to be the optimal constant in the
Sobolev inequality on $\mathbb{R}^{n}$, that is $A=\frac{4}{n\left(
n-2\right) \omega _{n}^{2/n}},$ where $\omega _{n}$ is the volume of the
unit sphere in $\mathbb{R}^{n+1}$. We refer to \cite{H} for an excellent
reference on the $AB$ program.

Our interest lies in the case when $B=1$, that is 
\begin{equation*}
\left( \int_{M}\left\vert \phi \right\vert ^{p}\right) ^{\frac{2}{p}%
}-\int_{M}\phi ^{2}\leq A\int_{M}\left\vert \nabla \phi \right\vert ^{2},
\end{equation*}%
for any $\phi \in C^{\infty }\left( M\right) $ and $\frac{2n}{n-2}\geq p>2$.
Here and below we normalized the volume so that $\mathrm{Vol}\left( M\right)
=1$. The best constant $A$ can be estimated geometrically in terms of a\
Ricci curvature lower bound on the manifold.

\begin{theorem}\label{theo intro}
\label{SR}Let $\left( M^{n},g\right) $ be a compact Riemannian manifold with
Ricci curvature $\mathrm{Ric}\geq \rho $, for some constant $\rho >0$. Then 
\begin{equation}\label{sobo intro}
\left( \int_{M}\left\vert \phi \right\vert ^{p}\right) ^{\frac{2}{p}%
}-\int_{M}\phi ^{2}\leq \frac{\left( n-1\right) \left( p-2\right) }{n\rho }%
\int_{M}\left\vert \nabla \phi \right\vert ^{2},
\end{equation}%
for all $\phi \in C^{\infty }\left( M\right) $.
\end{theorem}

This was proved by Bidaut-Veron and Veron in \cite{VV}, using a method
proposed in \cite{GS}. In this form, the inequality actually holds for $%
1\leq p\leq 2$ as well, which are called Beckner inequalities \cite{GZ}.
These interpolate between the Poincar\'e inequality for $p=1$ and the log
Sobolev inequality for $p\rightarrow 2$. 

On the class of Riemannian manifolds with $\mathrm{Ric}\geq \rho $, the constant $\frac{\left( n-1\right) \left( p-2\right) }{n\rho }$ in \eqref{sobo intro} can not be improved since it is optimal for the spheres, see for instance \cite{DEKL}.  Our goal in this paper is to improve the constant $\frac{\left( n-1\right) \left( p-2\right) }{n\rho }$ in Theorem \ref{theo intro} in the K\"{a}hler setting.
We have been able to do so for the entire range $1\leq p<\frac{2n}{n-2},$ by
using an improved integrated Bochner's inequality that takes into account the K\"{a}hler
structure (see Theorem \ref{improved CD}). The following is the main result of this paper.

\begin{theorem}
\label{SK}Let $\left( M^{m},g\right) $ be a compact K\"{a}hler manifold of complex
dimension $m\geq 2$ and with Ricci curvature $\mathrm{Ric}\geq \rho $ with $\rho >0$ and $%
\mathrm{Vol}\left( M\right) =1$. For any $2<p\leq \frac{2m}{m-1}$ we have
the Sobolev inequality 
\begin{equation*}
\left( \int_{M}\left\vert \phi \right\vert ^{p}\right) ^{\frac{2}{p}%
}-\int_{M}\phi ^{2}\leq C_{S}\int_{M}\left\vert \nabla \phi \right\vert ^{2},
\end{equation*}%
for any $\phi \in C^{\infty }\left( M\right) $, where 
\begin{equation*}
C_{S}=\frac{p-2}{\left( p-1\right) 2m\rho }\left( 2m+p+1-2\sqrt{\left(
m+1\right) \left( 2m-\left( m-1\right) p\right) }\right) .
\end{equation*}%
For any $1<p\leq 2$ we have the Beckner inequality 
\begin{equation*}
\int_{M}\phi ^{2}-\left( \int_{M}\phi ^{\frac{2}{p}}\right) ^{p}\leq
C_{B}\int_{M}\left\vert \nabla \phi \right\vert ^{2},
\end{equation*}%
for any $\phi \in C^{\infty }\left( M\right) $, where 
\begin{equation*}
C_{B}=\frac{p-1}{p}\frac{2m}{\left( \left( m-1\right) p+2\right) \rho }.
\end{equation*}
\end{theorem}

Let us note that Sobolev inequality on K\"{a}hler manifolds improves that in
Theorem \ref{SR} for all $1 \le p<\frac{2n}{n-2}=\frac{2m}{m-1}$. In the
critical case $p=\frac{2n}{n-2}$ the optimal Sobolev constant is related to the
Yamabe invariant, and so it cannot be improved in the K\"{a}hler setting.
Indeed, by Obata theorem \cite{O}, any non-spherical Einstein metric is the
unique constant scalar curvature metric in its conformal class, which
implies that K\"{a}hler Einstein metrics have the same Sobolev constants as
spheres.

The Beckner inequalities also refine those from Theorem \ref{SR}. In
particular, letting $p\rightarrow 1$, we have a log-Sobolev inequality of the
form 
\begin{equation*}
\int_{M}\phi ^{2}\ln \phi ^{2}-\int_{M}\phi ^{2}\ln \int_{M}\phi ^{2}\leq 
\frac{2m}{\left( m+1\right) \rho }\int_{M}\left\vert \nabla \phi \right\vert
^{2},
\end{equation*}%
and for $p=2$ a sharp Poincar\'{e} inequality of the form 
\begin{equation*}
\int_{M}\phi ^{2}-\left( \int_{M}\phi \right) ^{2}\leq \frac{1}{2\rho }%
\int_{M}\left\vert \nabla \phi \right\vert ^{2}.
\end{equation*}
It is not clear to us at this point if the Beckner-Sobolev
constants we obtain  are optimal for other values of $p$ than 2. Let us however point out that according to Theorem 1.3 in \cite{SC},  for the log-Sobolev inequality $p=1$, the constant $\frac{2m}{\left( m+1\right) \rho }$ is asymptotically sharp in the case of the complex projective space $\mathbb{CP}^m$ when $m \to +\infty$.

As an application of these results, we use a theory developed by Bakry and
Ledoux \ \cite{BL} to estimate the diameter of such manifolds. The
Bonnet-Myers diameter estimate asserts that a complete $n$-dimensional Riemannian
manifold $\left( M^{n},g\right) $ with Ricci curvature lower bound $\mathrm{%
Ric}\geq n-1$ is compact and has a sharp diameter upper bound $\mathrm{diam}%
\left( M\right) \leq \pi $. By Cheng's theorem, equality is achieved if $M$
is a sphere of constant sectional curvature. Since the rigidity in
Bonnet-Myers theorem does not apply to K\"{a}hler manifolds for complex
dimension $m\geq 2$, it is a natural question to study the diameter bound in
this setting. Several results are known. If $\left( M^{m},g\right) $ has
bisectional curvature bounded by $BK\geq 1,$ it is known \cite{LW} that $%
\mathrm{diam}\left( M\right) \leq \frac{\pi }{2}$, this result is sharp and
equality is achieved by the complex projective space $\mathbb{CP}^{m}$.
Rigidity in this estimate has been studied recently in \cite{TY, LY}. In
fact, it should be noted that this sharp estimate holds under the weaker
lower bound on holomorphic sectional curvature, \cite{T}.

On the other hand, one cannot replace the bisectional curvature by Ricci
curvature in the K\"{a}hler case, because the K\"{a}hler-Einstein metric on $%
\mathbb{CP}^{1}\times $ $\mathbb{CP}^{1}\times ..\times $ $\mathbb{CP}^{1}$
has larger diameter than $\mathbb{CP}^{m}$. A sharp diameter estimate for K%
\"{a}hler manifolds $\left( M^{m},g\right) $ with $\mathrm{Ric}\geq 2m-1$ is
currently not known, but Liu proved \cite{L} that there exists a constant $%
\varepsilon \left( m\right) $ depending only on $m$ so that $\mathrm{diam}%
\left( M\right) \leq \pi -\varepsilon \left( m\right) $, for any K\"{a}hler
manifold satisfying $\mathrm{Ric}\geq 2m-1.$ However, the precise dependency
of $\varepsilon $ on the dimension $m$ in \cite{L} is somewhat difficult to
state.

As an application of Theorem \ref{SK} we obtain the following.

\begin{theorem}
\label{A}Let $\left( M^{m},g\right) $ be a K\"{a}hler manifold of complex
dimension $m\geq 2$ and with Ricci curvature $\mathrm{Ric}\geq 2m-1$. Then 
\begin{equation*}
\mathrm{diam}\left( M\right) \leq \pi \left( 1-\frac{1}{24m}\right).
\end{equation*}
\end{theorem}

The organization of the paper is as follows. In Section \ref{Ineq} we
present a new Bochner type argument for K\"{a}hler manifolds. This is used
in Section \ref{BS} to prove Theorem \ref{SK}. Finally, the diameter
estimate is proved in Section \ref{Diam}.

\

\textbf{Acknowledgements.} It is our pleasure to thank Xiaodong Wang and
Jiaping Wang for useful discussions and for their interest in this work.

\section{A differential inequality for K\"{a}hler manifolds\label{Ineq}}

Let us set the notation that will be used throughout this note. Let $\left(
M^{m},g\right) $ be a compact K\"{a}hler manifold of complex dimension $m$. The K%
\"{a}hler metric $ds^{2}=g_{\alpha \bar{\beta}}dz^{\alpha }d\bar{z}^{\beta }$
defines a Riemannian metric by $\mathrm{Re}\left( ds^{2}\right) $. Hence, if 
$\left\{ e_{k}\right\} _{k=1,2m}$ is an orthonormal frame so that $%
e_{2k}=Je_{2k-1}$, then $\left\{ \nu _{\alpha }\right\} _{\alpha =1,m}$ 
\begin{equation*}
\nu _{\alpha }=\frac{1}{2}\left( e_{2\alpha -1}-\sqrt{-1}e_{2\alpha }\right)
\end{equation*}%
is a unitary frame. It follows that in this normalization of the Riemannian
metric, we have 
\begin{eqnarray*}
\Delta f &=&4f_{\alpha \bar{\alpha}} \\
\left\langle \nabla f,\nabla h\right\rangle &=&2\left( f_{\alpha }h_{\bar{%
\alpha}}+f_{\bar{\alpha}}h_{\alpha }\right) .
\end{eqnarray*}%
A lower bound for Ricci curvature $\mathrm{Ric}\geq \rho $ means that $%
\mathrm{Ric}\left( e_{j},e_{k}\right) \geq \rho \delta _{jk}$, or in the
unitary frame $\left\{ \nu _{\alpha }\right\} _{\alpha =1,m}$ that $%
R_{\alpha \bar{\beta}}\geq \frac{1}{2}\rho \delta _{\alpha \bar{\beta}}$.

Let $a\in \mathbb{R}$ and $k,q\geq 0$ be arbitrary. We fix the following
constants 
\begin{eqnarray}
A_{1} &=&2\left( m+1\right) ak-\left( 1+2k\right) mq  \label{A1} \\
B_{1} &=&\left( \left( m-1\right) +mk\right) ka^{2}+\left( 2-\left(
k+1\right) q\right) mka  \notag \\
&&+\frac{m}{4}\left( 1-k\right) ^{2}q^{2}+kmq\left( q-1\right) ,  \notag
\end{eqnarray}%
and 
\begin{eqnarray}
A_{2} &=&2\left( m+1\right) ak+\frac{1}{2}\left( 1-k\right) mq-\frac{3}{2}kq
\label{A2} \\
B_{2} &=&\left( \left( m-1\right) +mk\right) ka^{2}+\left( 2-\left(
k+1\right) q\right) mka  \notag \\
&&+\frac{m}{4}\left( 1-k\right) ^{2}q^{2}+\frac{q\left( q-1\right) }{2}%
\left( m\left( k-1\right) +k\right) .  \notag
\end{eqnarray}%
We have the following integral estimates which are the keys to our results and improve upon the integrated curvature dimension type inequality available on arbitrary Riemannian manifolds.

\begin{theorem}\label{improved CD}
\label{Bochner}Let $\left( M^{m},g\right) $ be a compact K\"{a}hler manifold of
complex dimension $m\geq 2$ and with Ricci curvature $\mathrm{Ric}\geq \rho$ for some $\rho >0$. For any function $u>0,$ any $a\in \mathbb{R}$ and $k,q\geq 0,$ we have 
\begin{eqnarray*}
\left( m+\left( m-1\right) k\right) \int_{M}\left( \Delta u\right) ^{2}u^{q}
&\geq &2m\rho \int_{M}\left\vert \nabla u\right\vert
^{2}u^{q}+A_{1}\int_{M}\left\vert \nabla u\right\vert ^{2}\left( \Delta
u\right) u^{q-1} \\
&&-B_{1}\int_{M}\left\vert \nabla u\right\vert ^{4}u^{q-2},
\end{eqnarray*}%
where $A_{1}$ and $B_{1}$ are specified in (\ref{A1}). If $0\leq k\leq \frac{%
m}{m-1},$ then for any $u>0$ and any $q\geq 0$ we have%
\begin{eqnarray*}
\left( m+\left( m-1\right) k\right) \int_{M}\left\vert \nabla
^{2}u\right\vert ^{2}u^{q} &\geq &\left( m-\left( m-1\right) k\right) \rho
\int_{M}\left\vert \nabla u\right\vert ^{2}u^{q}+A_{2}\int_{M}\left\vert
\nabla u\right\vert ^{2}\left( \Delta u\right) u^{q-1} \\
&&-B_{2}\int_{M}\left\vert \nabla u\right\vert ^{4}u^{q-2},
\end{eqnarray*}%
for $A_{2}$ and $B_{2}$ specified in (\ref{A2}).
\end{theorem}

\begin{proof}
It suffices to prove the theorem for any $k,q>0$, and any $a\in \mathbb{R}$.
We have by the arithmetic-mean inequality that 
\begin{equation*}
\left\vert u_{\alpha \bar{\beta}}+a\frac{u_{\alpha }u_{\bar{\beta}}}{u}%
\right\vert ^{2}\geq \frac{1}{16m}\left\vert \Delta u+a\frac{\left\vert
\nabla u\right\vert ^{2}}{u}\right\vert ^{2}.
\end{equation*}%
Expanding the terms, this implies 
\begin{eqnarray*}
&&\left\vert u_{\alpha \bar{\beta}}\right\vert ^{2}+2a\frac{u_{\alpha \bar{%
\beta}}u_{\bar{\alpha}}u_{\beta }}{u}+\frac{a^{2}}{16}\frac{m-1}{m}\frac{%
\left\vert \nabla u\right\vert ^{4}}{u^{2}} \\
&\geq &\frac{1}{16m}\left( \Delta u\right) ^{2}+\frac{a}{8m}\frac{\left\vert
\nabla u\right\vert ^{2}}{u}\Delta u.
\end{eqnarray*}%
Multiplying this with $u^{q}$ it follows that%
\begin{eqnarray}
\int_{M}\left\vert u_{\alpha \bar{\beta}}\right\vert ^{2}u^{q} &\geq &\frac{1%
}{16m}\int_{M}\left( \Delta u\right) ^{2}u^{q}+\frac{a}{8m}%
\int_{M}\left\vert \nabla u\right\vert ^{2}\left( \Delta u\right) u^{q-1}
\label{z1} \\
&&-2a\int_{M}\left( u_{\alpha \bar{\beta}}u_{\bar{\alpha}}u_{\beta }\right)
u^{q-1}  \notag \\
&&-\frac{a^{2}}{16}\frac{m-1}{m}\int_{M}\left\vert \nabla u\right\vert
^{4}u^{q-2}.  \notag
\end{eqnarray}%
For $b\in \mathbb{R}$ to be specified later we note the following inequality 
\begin{equation*}
\left\vert u_{\alpha \beta }\right\vert ^{2}+2b\frac{\mathrm{Re}\left(
u_{\alpha \beta }u_{\bar{\alpha}}u_{\bar{\beta}}\right) }{u}+\frac{b^{2}}{16}%
\frac{\left\vert \nabla u\right\vert ^{4}}{u^{2}}\geq 0,
\end{equation*}%
where $\mathrm{Re}\left( w\right) =\frac{1}{2}\left( w+\bar{w}\right) $
denotes the real part of $w$. We get that 
\begin{equation}
\int_{M}\left\vert u_{\alpha \beta }\right\vert ^{2}u^{q}\geq -2b\int_{M}%
\mathrm{Re}\left( u_{\alpha \beta }u_{\bar{\alpha}}u_{\bar{\beta}}\right)
u^{q-1}-\frac{b^{2}}{16}\int_{M}\left\vert \nabla u\right\vert ^{4}u^{q-2}.
\label{z2}
\end{equation}%
Integrating by parts, we obtain the following 
\begin{eqnarray*}
-\int_{M}\left( u_{\alpha \beta }u_{\bar{\alpha}}u_{\bar{\beta}}\right)
u^{q-1} &=&\int_{M}u_{\alpha }\left( \left( u_{\bar{\alpha}}u_{\bar{\beta}%
}\right) u^{q-1}\right) _{\beta } \\
&=&\int_{M}\left( u_{\bar{\alpha}\beta }u_{\alpha }u_{\bar{\beta}}\right)
u^{q-1}+\frac{1}{16}\int_{M}\left\vert \nabla u\right\vert ^{2}\left( \Delta
u\right) u^{q-1} \\
&&+\frac{q-1}{16}\int_{M}\left\vert \nabla u\right\vert ^{4}u^{q-2}.
\end{eqnarray*}%
Hence, we have 
\begin{eqnarray}
-\int_{M}\mathrm{Re}\left( u_{\alpha \beta }u_{\bar{\alpha}}u_{\bar{\beta}%
}\right) u^{q-1} &=&\int_{M}\left( u_{\bar{\alpha}\beta }u_{\alpha }u_{\bar{%
\beta}}\right) u^{q-1}  \label{z3} \\
&&+\frac{1}{16}\int_{M}\left\vert \nabla u\right\vert ^{2}\left( \Delta
u\right) u^{q-1}  \notag \\
&&+\frac{q-1}{16}\int_{M}\left\vert \nabla u\right\vert ^{4}u^{q-2}.  \notag
\end{eqnarray}%
Using this in (\ref{z2}) proves that 
\begin{eqnarray}
\int_{M}\left\vert u_{\alpha \beta }\right\vert ^{2}u^{q} &\geq
&2b\int_{M}\left( u_{\bar{\alpha}\beta }u_{\alpha }u_{\bar{\beta}}\right)
u^{q-1}+\frac{b}{8}\int_{M}\left\vert \nabla u\right\vert ^{2}\left( \Delta
u\right) u^{q-1}  \label{z4} \\
&&-\left( \frac{b^{2}}{16}+\frac{b\left( 1-q\right) }{8}\right)
\int_{M}\left\vert \nabla u\right\vert ^{4}u^{q-2}.  \notag
\end{eqnarray}%
Multiply (\ref{z4}) by $\frac{1}{k}>0$ and add to (\ref{z1}) to conclude%
\begin{eqnarray}
&&\int_{M}\left\vert u_{\alpha \bar{\beta}}\right\vert ^{2}u^{q}+\frac{1}{k}%
\int_{M}\left\vert u_{\alpha \beta }\right\vert ^{2}u^{q}  \label{z5} \\
&\geq &\frac{1}{16m}\int_{M}\left( \Delta u\right) ^{2}u^{q}+\left( \frac{a}{%
8m}+\frac{b}{8k}\right) \int_{M}\left\vert \nabla u\right\vert ^{2}\left(
\Delta u\right) u^{q-1}  \notag \\
&&-2\left( a-\frac{b}{k}\right) \int_{M}\left( u_{\alpha \bar{\beta}}u_{\bar{%
\alpha}}u_{\beta }\right) u^{q-1}  \notag \\
&&-\left( \frac{m-1}{16m}a^{2}+\frac{b^{2}}{16k}+\frac{b\left( 1-q\right) }{%
8k}\right) \int_{M}\left\vert \nabla u\right\vert ^{4}u^{q-2}.  \notag
\end{eqnarray}%
We now integrate by parts and use Ricci identities to get 
\begin{eqnarray}
\int_{M}\left\vert u_{\alpha \bar{\beta}}\right\vert ^{2}u^{q}
&=&-\int_{M}\left( u_{\alpha \bar{\beta}\beta }u_{\bar{\alpha}}\right)
u^{q}-q\int_{M}\left( u_{\alpha \bar{\beta}}u_{\bar{\alpha}}u_{\beta
}\right) u^{q-1}  \label{z6} \\
&=&-\frac{1}{16}\int_{M}\left\langle \nabla \Delta u,\nabla u\right\rangle
u^{q}-q\int_{M}\left( u_{\alpha \bar{\beta}}u_{\bar{\alpha}}u_{\beta
}\right) u^{q-1}  \notag \\
&=&\frac{1}{16}\int_{M}\left( \Delta u\right) ^{2}u^{q}+\frac{q}{16}%
\int_{M}\left\vert \nabla u\right\vert ^{2}\left( \Delta u\right) u^{q-1} 
\notag \\
&&-q\int_{M}\left( u_{\alpha \bar{\beta}}u_{\bar{\alpha}}u_{\beta }\right)
u^{q-1}.  \notag
\end{eqnarray}%
From (\ref{z6}) we get 
\begin{eqnarray*}
\int_{M}\left( u_{\alpha \bar{\beta}}u_{\bar{\alpha}}u_{\beta }\right)
u^{q-1} &=&-\frac{1}{q}\int_{M}\left\vert u_{\alpha \bar{\beta}}\right\vert
^{2}u^{q}+\frac{1}{16q}\int_{M}\left( \Delta u\right) ^{2}u^{q} \\
&&+\frac{1}{16}\int_{M}\left\vert \nabla u\right\vert ^{2}\left( \Delta
u\right) u^{q-1},
\end{eqnarray*}%
that we plug into (\ref{z5}) to obtain 
\begin{eqnarray}
&&\int_{M}\left( 1-\frac{2}{q}\left( a-\frac{b}{k}\right) \right) \left\vert
u_{\alpha \bar{\beta}}\right\vert ^{2}u^{q}+\frac{1}{k}\int_{M}\left\vert
u_{\alpha \beta }\right\vert ^{2}u^{q}  \label{z7} \\
&\geq &\left( \frac{1}{16m}-\frac{1}{8q}\left( a-\frac{b}{k}\right) \right)
\int_{M}\left( \Delta u\right) ^{2}u^{q}  \notag \\
&&+\left( -\frac{m-1}{8m}a+\frac{b}{4k}\right) \int_{M}\left\vert \nabla
u\right\vert ^{2}\left( \Delta u\right) u^{q-1}  \notag \\
&&-\left( \frac{m-1}{16m}a^{2}+\frac{b^{2}}{16k}+\frac{b\left( 1-q\right) }{%
8k}\right) \int_{M}\left\vert \nabla u\right\vert ^{4}u^{q-2}.  \notag
\end{eqnarray}%
We now set 
\begin{equation*}
b=ka+\left( 1-k\right) \frac{q}{2},
\end{equation*}%
for which 
\begin{equation*}
1-\frac{2}{q}\left( a-\frac{b}{k}\right) =\frac{1}{k}.
\end{equation*}%
Noting that 
\begin{equation*}
\left\vert \nabla ^{2}u\right\vert ^{2}=8\left( \left\vert u_{\alpha \bar{%
\beta}}\right\vert ^{2}+\left\vert u_{\alpha \beta }\right\vert ^{2}\right) ,
\end{equation*}%
we get from (\ref{z7}) that 
\begin{eqnarray}
\int_{M}\left\vert \nabla ^{2}u\right\vert ^{2}u^{q} &\geq &\frac{m-\left(
m-1\right) k}{2m}\int_{M}\left( \Delta u\right) ^{2}u^{q}  \label{z8} \\
&&+A\int_{M}\left\vert \nabla u\right\vert ^{2}\left( \Delta u\right) u^{q-1}
\notag \\
&&-B\int_{M}\left\vert \nabla u\right\vert ^{4}u^{q-2},  \notag
\end{eqnarray}%
where 
\begin{eqnarray}
A &=&\frac{m+1}{m}ak+\left( 1-k\right) q  \label{z8.1} \\
B &=&\frac{1}{2}\left( \frac{m-1}{m}+k\right) ka^{2}+\left( \frac{q}{2}%
\left( 1-k\right) +\left( 1-q\right) \right) ka  \notag \\
&&+\frac{1}{8}\left( 1-k\right) ^{2}q^{2}+\frac{1}{2}\left( 1-k\right)
q\left( 1-q\right) .  \notag
\end{eqnarray}

We now use the Ricci curvature lower bound by integrating $%
\int_{M}\left\vert \nabla ^{2}u\right\vert ^{2}u^{q}$ by parts. \ From 
\begin{equation*}
\frac{1}{2}\Delta \left\vert \nabla u\right\vert ^{2}=\left\vert \nabla
^{2}u\right\vert ^{2}+\left\langle \nabla \Delta u,\nabla u\right\rangle +%
\mathrm{Ric}\left( \nabla u,\nabla u\right)
\end{equation*}%
we obtain%
\begin{eqnarray*}
\int_{M}\left\vert \nabla ^{2}u\right\vert ^{2}u^{q} &\leq
&-\int_{M}\left\langle \nabla \Delta u,\nabla u\right\rangle u^{q}+\frac{1}{2%
}\int_{M}\left( \Delta \left\vert \nabla u\right\vert ^{2}\right) u^{q} \\
&&-\rho \int_{M}\left\vert \nabla u\right\vert ^{2}u^{q} \\
&=&\int_{M}\left( \Delta u\right) ^{2}u^{q}+q\int_{M}\left\vert \nabla
u\right\vert ^{2}\left( \Delta u\right) u^{q-1} \\
&&-\frac{q}{2}\int_{M}\left\langle \nabla \left\vert \nabla u\right\vert
^{2},\nabla u\right\rangle u^{q-1}-\rho \int_{M}\left\vert \nabla
u\right\vert ^{2}u^{q} \\
&=&\int_{M}\left( \Delta u\right) ^{2}u^{q}+\frac{3q}{2}\int_{M}\left\vert
\nabla u\right\vert ^{2}\left( \Delta u\right) u^{q-1} \\
&&+\frac{q\left( q-1\right) }{2}\int_{M}\left\vert \nabla u\right\vert
^{4}u^{q-2}-\rho \int_{M}\left\vert \nabla u\right\vert ^{2}u^{q}.
\end{eqnarray*}%
This proves that 
\begin{eqnarray}
\int_{M}\left( \Delta u\right) ^{2}u^{q} &\geq &\int_{M}\left\vert \nabla
^{2}u\right\vert ^{2}u^{q}-\frac{3q}{2}\int_{M}\left\vert \nabla
u\right\vert ^{2}\left( \Delta u\right) u^{q-1}  \label{z10} \\
&&-\frac{q\left( q-1\right) }{2}\int_{M}\left\vert \nabla u\right\vert
^{4}u^{q-2}+\rho \int_{M}\left\vert \nabla u\right\vert ^{2}u^{q}.  \notag
\end{eqnarray}%
Plugging (\ref{z8}) into (\ref{z10}) we conclude that 
\begin{eqnarray}
\left( m+\left( m-1\right) k\right) \int_{M}\left( \Delta u\right) ^{2}u^{q}
&\geq &2m\rho \int_{M}\left\vert \nabla u\right\vert ^{2}u^{q}  \label{z11}
\\
&&+A_{1}\int_{M}\left\vert \nabla u\right\vert ^{2}\left( \Delta u\right)
u^{q-1}  \notag \\
&&-B_{1}\int_{M}\left\vert \nabla u\right\vert ^{4}u^{q-2},  \notag
\end{eqnarray}%
where%
\begin{eqnarray}
A_{1} &=&2\left( m+1\right) ak-\left( 1+2k\right) mq  \label{z11.1} \\
B_{1} &=&\left( \left( m-1\right) +mk\right) ka^{2}+\left( 2-\left(
k+1\right) q\right) mka  \notag \\
&&+\frac{1}{4}\left( 1-k\right) ^{2}mq^{2}+mkq\left( q-1\right) .  \notag
\end{eqnarray}%
We now assume that 
\begin{equation*}
k\leq \frac{m}{m-1}.
\end{equation*}%
Plugging (\ref{z11}) into (\ref{z8}) we now get%
\begin{eqnarray}
\left( m+\left( m-1\right) k\right) \int_{M}\left\vert \nabla
^{2}u\right\vert ^{2}u^{q} &\geq &\left( m-\left( m-1\right) k\right) \rho
\int_{M}\left\vert \nabla u\right\vert ^{2}u^{q}  \label{z12} \\
&&+A_{2}\int_{M}\left\vert \nabla u\right\vert ^{2}\left( \Delta u\right)
u^{q-1}  \notag \\
&&-B_{2}\int_{M}\left\vert \nabla u\right\vert ^{4}u^{q-2},  \notag
\end{eqnarray}%
where 
\begin{eqnarray}
A_{2} &=&2\left( m+1\right) ak+\frac{1}{2}\left( 1-k\right) mq-\frac{3}{2}kq
\label{z12.1} \\
B_{2} &=&\left( \left( m-1\right) +mk\right) ka^{2}+\left( 2-\left(
k+1\right) q\right) mka  \notag \\
&&+\frac{m}{4}\left( 1-k\right) ^{2}q^{2}+\frac{q\left( q-1\right) }{2}%
\left( m\left( k-1\right) +k\right) .  \notag
\end{eqnarray}%
By (\ref{z11}) and (\ref{z12}) we obtain the result.
\end{proof}

Let us note that for $k=1$ and $q=0$, (\ref{z11}) recovers the Bochner
inequality on p. 338 of \cite{Le}, and for $k=1$ (\ref{z12}) recovers the
Bochner inequality on p.13 of \cite{GZ}.

\section{Beckner-Sobolev inequalities \label{BS}}

Recall that on a Riemannian manifold $\left( M^{n},g\right) $ with Ricci
curvature bounded below by $\mathrm{Ric}\geq \rho $ and volume normalized by 
$\mathrm{Vol}\left( M\right) =1,$ the following sharp Sobolev inequality
holds \cite{Le, VV}%
\begin{equation}
\left( \int_{M}\left\vert \phi \right\vert ^{p}\right) ^{\frac{2}{p}%
}-\int_{M}\phi ^{2}\leq \frac{\left( n-1\right) \left( p-2\right) }{n\rho }%
\int_{M}\left\vert \nabla \phi \right\vert ^{2},  \label{a1}
\end{equation}%
for any $\phi \in C^{\infty }\left( M\right) $ and $2<p\leq \frac{2n}{n-2}$.
Our goal is to improve (\ref{a1}) in the K\"{a}hler setting. As a
preparation, we prove the following result.

\begin{lemma}
\label{B}Let $\left( M^{m},g\right) $ be a compact K\"{a}hler manifold
satisfying $\mathrm{Ric}\geq \rho $ with $\rho >0$. Assume there exists a smooth
nonconstant positive solution $f$ of 
\begin{equation*}
\Delta f=C\left( f-f^{p-1}\right) ,
\end{equation*}%
for some $2<p\leq \frac{2m}{m-1}$ and some constant $C\in \mathbb{R}$. Then 
\begin{equation*}
C\geq \frac{2m\rho }{\left( m+\left( m-1\right) k\right) \left( p-2\right) },
\end{equation*}%
for any $k>0$ such that 
\begin{equation*}
p\leq 1+\frac{m+1}{m-1}\frac{4k}{\left( k+1\right) ^{2}}.
\end{equation*}
\end{lemma}

\begin{proof}
We follow the approach in \cite{VV}, see also \cite{Ba, Le} and Ch.6.8 of 
\cite{BGL}. Define 
\begin{equation*}
u=f^{\frac{1}{r}}
\end{equation*}%
for some $r\in \mathbb{R}$ to be determined later.

We claim the following identity 
\begin{eqnarray}
C\left( p-2\right) \int_{M}\left\vert \nabla u\right\vert ^{2}
&=&\int_{M}\left( \Delta u\right) ^{2}+r\left( p-1\right) \int_{M}\frac{%
\left\vert \nabla u\right\vert ^{2}}{u}\Delta u  \label{a3} \\
&&+\left( r-1\right) \left( r\left( p-2\right) +1\right) \int_{M}\frac{%
\left\vert \nabla u\right\vert ^{4}}{u^{2}}.  \notag
\end{eqnarray}%
Indeed, since $f=u^{r}$ and $\Delta f=C\left( f-f^{p-1}\right) ,$ we get the
following equation for $u$ 
\begin{equation}
ru^{r-1}\Delta u+r\left( r-1\right) u^{r-2}\left\vert \nabla u\right\vert
^{2}=Cu^{r}-Cu^{r\left( p-1\right) }.  \label{a4}
\end{equation}%
Multiply (\ref{a4}) by $\left\vert \nabla u\right\vert ^{2}u^{-r}$ to get 
\begin{equation}
r\int_{M}\frac{\left\vert \nabla u\right\vert ^{2}}{u}\Delta u+r\left(
r-1\right) \int_{M}\frac{\left\vert \nabla u\right\vert ^{4}}{u^{2}}%
=C\int_{M}\left\vert \nabla u\right\vert ^{2}-C\int_{M}u^{r\left( p-2\right)
}\left\vert \nabla u\right\vert ^{2}.  \label{a5}
\end{equation}%
We can write using integration by parts 
\begin{eqnarray*}
\int_{M}u^{r\left( p-2\right) }\left\vert \nabla u\right\vert ^{2} &=&\frac{1%
}{r\left( p-2\right) +1}\int_{M}\left\langle \nabla u,\nabla u^{r\left(
p-2\right) +1}\right\rangle \\
&=&-\frac{1}{r\left( p-2\right) +1}\int_{M}u^{r\left( p-2\right) +1}\Delta u.
\end{eqnarray*}%
Hence, (\ref{a5}) becomes 
\begin{eqnarray}
&&r\int_{M}\frac{\left\vert \nabla u\right\vert ^{2}}{u}\Delta u+r\left(
r-1\right) \int_{M}\frac{\left\vert \nabla u\right\vert ^{4}}{u^{2}}
\label{a6} \\
&=&C\int_{M}\left\vert \nabla u\right\vert ^{2}+\frac{C}{r\left( p-2\right)
+1}\int_{M}u^{r\left( p-2\right) +1}\Delta u.  \notag
\end{eqnarray}%
To compute the second term on the right side of (\ref{a6}) we multiply (\ref%
{a4}) by $u^{-r+1}\Delta u$ and get 
\begin{equation*}
r\int_{M}\left( \Delta u\right) ^{2}+r\left( r-1\right) \int_{M}\frac{%
\left\vert \nabla u\right\vert ^{2}}{u}\Delta u=-C\int_{M}\left\vert \nabla
u\right\vert ^{2}-C\int_{M}u^{r\left( p-2\right) +1}\Delta u.
\end{equation*}%
This proves that 
\begin{eqnarray*}
\frac{C}{r\left( p-2\right) +1}\int_{M}u^{r\left( p-2\right) +1}\Delta u &=&-%
\frac{r}{r\left( p-2\right) +1}\int_{M}\left( \Delta u\right) ^{2}-\frac{%
r\left( r-1\right) }{r\left( p-2\right) +1}\int_{M}\frac{\left\vert \nabla
u\right\vert ^{2}}{u}\Delta u \\
&&-\frac{C}{r\left( p-2\right) +1}\int_{M}\left\vert \nabla u\right\vert
^{2}.
\end{eqnarray*}%
Plugging this in (\ref{a6}) implies that 
\begin{eqnarray*}
&&r\left( p-1\right) \int_{M}\frac{\left\vert \nabla u\right\vert ^{2}}{u}%
\Delta u+\left( r-1\right) \left( r\left( p-2\right) +1\right) \int_{M}\frac{%
\left\vert \nabla u\right\vert ^{4}}{u^{2}} \\
&=&C\left( p-2\right) \int_{M}\left\vert \nabla u\right\vert
^{2}-\int_{M}\left( \Delta u\right) ^{2},
\end{eqnarray*}%
which is exactly (\ref{a3}).\ 

Applying Theorem \ref{Bochner} for $q=0$ we get

\begin{eqnarray}
\rho \int_{M}\left\vert \nabla u\right\vert ^{2} &\leq &\frac{1}{2m}\left(
\left( m-1\right) k+m\right) \int_{M}\left( \Delta u\right) ^{2}  \label{a11}
\\
&&-\frac{a\left( m+1\right) k}{m}\int_{M}\frac{\left\vert \nabla
u\right\vert ^{2}}{u}\Delta u  \notag \\
&&+\left( ak+\frac{a^{2}k}{2m}\left( mk+\left( m-1\right) \right) \right)
\int_{M}\frac{\left\vert \nabla u\right\vert ^{4}}{u^{2}}.  \notag
\end{eqnarray}%
Note that (\ref{a3}) yields 
\begin{eqnarray*}
-\frac{a\left( m+1\right) k}{m}\int_{M}\frac{\left\vert \nabla u\right\vert
^{2}}{u}\Delta u &=&-\frac{C\left( p-2\right) }{r\left( p-1\right) }\frac{%
a\left( m+1\right) k}{m}\int_{M}\left\vert \nabla u\right\vert ^{2} \\
&&+\frac{1}{r\left( p-1\right) }\frac{a\left( m+1\right) k}{m}\int_{M}\left(
\Delta u\right) ^{2} \\
&&+\frac{\left( r-1\right) \left( r\left( p-2\right) +1\right) }{r\left(
p-1\right) }\frac{a\left( m+1\right) k}{m}\int_{M}\frac{\left\vert \nabla
u\right\vert ^{4}}{u^{2}}.
\end{eqnarray*}%
Using this in (\ref{a11}) implies%
\begin{eqnarray}
&&\left( \rho +\frac{C\left( p-2\right) }{r\left( p-1\right) }\frac{a\left(
m+1\right) k}{m}\right) \int_{M}\left\vert \nabla u\right\vert ^{2}
\label{a12} \\
&\leq &\left( \frac{1}{2m}\left( \left( m-1\right) k+m\right) +\frac{1}{%
r\left( p-1\right) }\frac{a\left( m+1\right) k}{m}\right) \int_{M}\left(
\Delta u\right) ^{2}  \notag \\
&&+\mathcal{F}\int_{M}\frac{\left\vert \nabla u\right\vert ^{4}}{u^{2}}, 
\notag
\end{eqnarray}%
where%
\begin{equation}
\mathcal{F=}ak+\frac{a^{2}k}{2m}\left( mk+\left( m-1\right) \right) +\frac{%
\left( r-1\right) \left( r\left( p-2\right) +1\right) }{r\left( p-1\right) }%
\frac{a\left( m+1\right) k}{m}.  \label{a12.1}
\end{equation}%
We now set 
\begin{equation}
a=-\frac{\left( p-1\right) r\left( m+\left( m-1\right) k\right) }{2\left(
m+1\right) k},  \label{a13}
\end{equation}%
so that 
\begin{equation*}
\frac{1}{2m}\left( \left( m-1\right) k+m\right) +\frac{1}{r\left( p-1\right) 
}\frac{a\left( m+1\right) k}{m}=0.
\end{equation*}%
Hence, (\ref{a12}) becomes 
\begin{equation}
\left( \rho -\frac{\left( p-2\right) \left( m+\left( m-1\right) k\right) }{2m%
}C\right) \int_{M}\left\vert \nabla u\right\vert ^{2}\leq \mathcal{F}\int_{M}%
\frac{\left\vert \nabla u\right\vert ^{4}}{u^{2}}.  \label{a13.1}
\end{equation}%
By plugging (\ref{a13}) in (\ref{a12.1}) we get%
\begin{equation*}
\mathcal{F}=\frac{m+\left( m-1\right) k}{8m\left( m+1\right) ^{2}k}\Theta
\left( r\right) ,
\end{equation*}%
where 
\begin{eqnarray}
\Theta \left( r\right) &=&-4km\left( m+1\right) \left( p-1\right) r
\label{a14} \\
&&+\left( p-1\right) ^{2}\left( m+\left( m-1\right) k\right) \left(
mk+\left( m-1\right) \right) r^{2}  \notag \\
&&-4k\left( m+1\right) ^{2}\left( r-1\right) \left( r\left( p-2\right)
+1\right)  \notag
\end{eqnarray}%
In conclusion, (\ref{a13.1}) implies that 
\begin{eqnarray}
&&-\frac{m+\left( m-1\right) k}{8m\left( m+1\right) ^{2}k}\Theta \left(
r\right) \int_{M}\frac{\left\vert \nabla u\right\vert ^{4}}{u^{2}}
\label{a14.1} \\
&\leq &\left( \frac{\left( m+\left( m-1\right) k\right) \left( p-2\right) }{%
2m}C-\rho \right) \int_{M}\left\vert \nabla u\right\vert ^{2}  \notag
\end{eqnarray}%
where $\Theta \left( r\right) $ is given in (\ref{a14}).

We write 
\begin{equation}
\Theta \left( r\right) =c_{0}r^{2}+c_{1}r+c_{2},  \label{a15}
\end{equation}%
\newline
where 
\begin{eqnarray*}
c_{0} &=&\left( p-1\right) ^{2}\left( m+\left( m-1\right) k\right) \left(
mk+\left( m-1\right) \right) -4k\left( m+1\right) ^{2}\left( p-2\right)  \\
c_{1} &=&-4k\left( m+1\right) \left( 2m+3-p\right)  \\
c_{2} &=&4k\left( m+1\right) ^{2}
\end{eqnarray*}%
We claim there exists $r\in \mathbb{R}$ so that $\Theta \left( r\right) =0$.
Indeed, observe that 
\begin{equation}
\left( c_{1}\right) ^{2}-4c_{0}c_{2}=16k\left( m+1\right) ^{2}Q  \label{a17}
\end{equation}%
where 
\begin{eqnarray*}
Q &=&k\left( 2m+3-p\right) ^{2}-\left( p-1\right) ^{2}\left( m+\left(
m-1\right) k\right) \left( mk+\left( m-1\right) \right)  \\
&&+4k\left( m+1\right) ^{2}\left( p-2\right)  \\
&=&m\left( m-1\right) \left( k+1\right) ^{2}\left( p-1\right) \left( 1+\frac{%
m+1}{m-1}\frac{4k}{\left( k+1\right) ^{2}}-p\right)  \\
&\geq &0.
\end{eqnarray*}%
In the last line we used the hypothesis that 
\begin{equation*}
p\leq 1+\frac{m+1}{m-1}\frac{4k}{\left( k+1\right) ^{2}}.
\end{equation*}%
This proves that indeed there exists $r\in \mathbb{R}$ so that $\Theta
\left( r\right) =0$. From (\ref{a14.1}) we get that 
\begin{equation*}
C\geq \frac{2m\rho }{\left( m+\left( m-1\right) k\right) \left( p-2\right) },
\end{equation*}%
which proves the lemma.
\end{proof}

From \cite{VV} and Lemma \ref{B} we obtain the following Sobolev inequality.

\begin{proposition}
\label{C}Let $\left( M^{m},g\right) $ be a compact K\"{a}hler manifold
satisfying $\mathrm{Ric}\geq \rho $ and $\mathrm{Vol}\left( M\right) =1.$
For any $2<p\leq \frac{2m}{m-1}$ we have the Sobolev inequality 
\begin{equation}
\left( \int_{M}\left\vert \phi \right\vert ^{p}\right) ^{\frac{2}{p}%
}-\int_{M}\phi ^{2}\leq \frac{\left( m+\left( m-1\right) k\right) \left(
p-2\right) }{2m\rho }\int_{M}\left\vert \nabla \phi \right\vert ^{2},
\label{a18}
\end{equation}%
for any $k>0$ such that 
\begin{equation}
p\leq 1+\frac{m+1}{m-1}\frac{4k}{\left( k+1\right) ^{2}}.  \label{a18'}
\end{equation}
\end{proposition}

\begin{proof}
Consider the functional 
\begin{equation*}
\mathcal{F}\left( \phi \right) =\frac{\left( \int_{M}\left\vert \phi
\right\vert ^{p}\right) ^{\frac{2}{p}}-\int_{M}\phi ^{2}}{\int_{M}\left\vert
\nabla \phi \right\vert ^{2}}.
\end{equation*}%
According to (\ref{a1}), 
\begin{equation*}
\sup_{\phi \in C^{\infty }\left( M\right) }\mathcal{F}\left( \phi \right)
\leq \frac{\left( n-1\right) \left( p-2\right) }{n\rho }.
\end{equation*}%
To improve (\ref{a1}), let $f$ be an extremum of the functional $\mathcal{F}$%
. By the argument in Ch. 6.8.2 of \ \cite{BGL}, we may assume that $f$ is
positive, nonconstant and smooth. Denoting with 
\begin{equation*}
\Lambda :=\sup_{\phi \in C^{\infty }\left( M\right) }\mathcal{F}\left( \phi
\right) ,
\end{equation*}%
we may normalize $f$ so that it satisfies the partial differential equation 
\begin{equation}
\Delta f=\frac{1}{\Lambda }\left( f-f^{p-1}\right) .  \label{a2}
\end{equation}%
The result now follows from Lemma \ref{B}.
\end{proof}

Let us note that solving the equation 
\begin{equation*}
p=1+\frac{m+1}{m-1}\frac{4k}{\left( k+1\right) ^{2}}
\end{equation*}%
in $k$ and plugging it in (\ref{a18}) implies that 
\begin{equation*}
\left( \int_{M}\left\vert \phi \right\vert ^{p}\right) ^{\frac{2}{p}%
}-\int_{M}\phi ^{2}\leq C_{S}\int_{M}\left\vert \nabla \phi \right\vert ^{2},
\end{equation*}%
for any $\phi \in C^{\infty }\left( M\right) $, where 
\begin{equation*}
C_{S}=\frac{p-2}{\left( p-1\right) 2m\rho }\left( 2m+p+1-2\sqrt{\left(
m+1\right) \left( 2m-\left( m-1\right) p\right) }\right) .
\end{equation*}%
This proves the first part of Theorem \ref{SK}.

As another consequence of Lemma \ref{B}, we obtain the following Beckner
type inequality, that completes the proof of Theorem \ref{SK}.

\begin{proposition}
\label{D}Let $\left( M^{m},g\right) $ be a K\"{a}hler manifold of complex
dimension $m\geq 2$ and with Ricci curvature $\mathrm{Ric}\geq \rho $ and $%
\mathrm{Vol}\left( M\right) =1.$ For any $1<p\leq 2$ we have the inequality 
\begin{equation*}
\int_{M}\phi ^{2}-\left( \int_{M}\phi ^{\frac{2}{p}}\right) ^{p}\leq
C_{B}\int_{M}\left\vert \nabla \phi \right\vert ^{2},
\end{equation*}%
where 
\begin{equation*}
C_{B}=\frac{p-1}{p}\frac{2m}{\left( \left( m-1\right) p+2\right) \rho }.
\end{equation*}
\end{proposition}

\begin{proof}
We follow the proof in \cite{GZ}. For a fixed $\phi $, let $f\left( t\right) 
$ be the solution of 
\begin{eqnarray*}
f_{t} &=&\Delta f \\
f\left( 0\right) &=&\phi ^{\frac{2}{p}}.
\end{eqnarray*}%
Define 
\begin{equation}
\Lambda \left( t\right) =\int_{M}f^{p}-\left( \int_{M}f\right) ^{p}.
\label{e1}
\end{equation}%
Our goal is to establish a differential inequality for $\Lambda \left(
t\right) $, for all $t>0$. For this, it is convenient to denote 
\begin{eqnarray}
q &=&\frac{2-p}{p-1}  \label{e2} \\
u &=&f^{p-1}.  \notag
\end{eqnarray}

Clearly, $q\geq 0$ and $u>0$. We have the following identities (cf. \cite{GZ}%
)%
\begin{equation}
\frac{p-1}{p}\Lambda ^{\prime }\left( t\right) =-\int_{M}\left\vert \nabla
u\right\vert ^{2}u^{q}  \label{e3}
\end{equation}%
and 
\begin{eqnarray}
\frac{p-1}{p}\Lambda ^{\prime \prime }\left( t\right) &\geq
&2\int_{M}\left\vert \nabla ^{2}u\right\vert ^{2}u^{q}+2\rho
\int_{M}\left\vert \nabla u\right\vert ^{2}u^{q}  \label{e4} \\
&&+q\int_{M}\left\vert \nabla u\right\vert ^{4}u^{q-2}.  \notag
\end{eqnarray}%
Indeed, we have 
\begin{eqnarray*}
\Lambda ^{\prime }\left( t\right) &=&p\int_{M}\left( f_{t}\right)
f^{p-1}-p\left( \int_{M}f\right) ^{p-1}\int_{M}f_{t} \\
&=&p\int_{M}\left( \Delta f\right) f^{p-1} \\
&=&-p\left( p-1\right) \int_{M}\left\vert \nabla f\right\vert ^{2}f^{p-2}.
\end{eqnarray*}%
Using that $f=u^{\frac{1}{p-1}}$, this immediately implies (\ref{e3}). Note
that $u$ verifies the equation 
\begin{equation*}
u_{t}=\Delta u+q\left\vert \nabla u\right\vert ^{2}u^{-1}.
\end{equation*}%
Therefore, 
\begin{equation}
\frac{d}{dt}\left\vert \nabla u\right\vert ^{2}=2\left\langle \nabla
u,\nabla \Delta u\right\rangle +2q\left\langle \nabla u,\nabla \left\vert
\nabla u\right\vert ^{2}\right\rangle u^{-1}-2q\left\vert \nabla
u\right\vert ^{4}u^{-2}.  \label{e5}
\end{equation}%
The Bochner formula yields 
\begin{equation*}
\Delta \left\vert \nabla u\right\vert ^{2}\geq 2\left\vert \nabla
^{2}u\right\vert ^{2}+2\rho \left\vert \nabla u\right\vert
^{2}+2\left\langle \nabla \Delta u,\nabla u\right\rangle .
\end{equation*}%
Combining with (\ref{e5}) \ we have 
\begin{eqnarray*}
\frac{d}{dt}\left\vert \nabla u\right\vert ^{2} &\leq &\Delta \left\vert
\nabla u\right\vert ^{2}-2\left\vert \nabla ^{2}u\right\vert ^{2}-2\rho
\left\vert \nabla u\right\vert ^{2} \\
&&+2q\left\langle \nabla u,\nabla \left\vert \nabla u\right\vert
^{2}\right\rangle u^{-1}-2q\left\vert \nabla u\right\vert ^{4}u^{-2}.
\end{eqnarray*}%
Now take a derivative in $t$ of (\ref{e3}) and integrate by parts to get 
\begin{eqnarray*}
\frac{p-1}{p}\Lambda ^{\prime \prime }\left( t\right) &=&-\int_{M}\frac{d}{dt%
}\left( \left\vert \nabla u\right\vert ^{2}\right) u^{q}-q\int_{M}\left\vert
\nabla u\right\vert ^{2}\left( \Delta u+q\left\vert \nabla u\right\vert
^{2}u^{-1}\right) u^{q-1} \\
&\geq &2\int_{M}\left\vert \nabla ^{2}u\right\vert ^{2}u^{q}+2\rho
\int_{M}\left\vert \nabla u\right\vert ^{2}u^{q}+q\int_{M}\left\vert \nabla
u\right\vert ^{4}u^{q-2},
\end{eqnarray*}%
which proves (\ref{e4}).

According to Lemma \ref{B} we have for any $0\leq k\leq \frac{m}{m-1}$, 
\begin{eqnarray*}
\left( m+\left( m-1\right) k\right) \int_{M}\left\vert \nabla
^{2}u\right\vert ^{2}u^{q} &\geq &\left( m-\left( m-1\right) k\right) \rho
\int_{M}\left\vert \nabla u\right\vert ^{2}u^{q}+A_{2}\int_{M}\left\vert
\nabla u\right\vert ^{2}\left( \Delta u\right) u^{q-1} \\
&&-B_{2}\int_{M}\left\vert \nabla u\right\vert ^{4}u^{q-2},
\end{eqnarray*}%
where 
\begin{eqnarray}
A_{2} &=&2\left( m+1\right) ak+\frac{1}{2}\left( 1-k\right) mq-\frac{3}{2}kq
\label{e6} \\
B_{2} &=&\left( \left( m-1\right) +mk\right) ka^{2}+\left( 2-\left(
k+1\right) q\right) mka  \notag \\
&&+\frac{m}{4}\left( 1-k\right) ^{2}q^{2}+\frac{q\left( q-1\right) }{2}%
\left( m\left( k-1\right) +k\right) .  \notag
\end{eqnarray}%
Let 
\begin{equation}
a=\frac{3q}{4\left( m+1\right) }-\frac{m}{4\left( m+1\right) }\frac{1-k}{k}q,
\label{e7}
\end{equation}%
for which $A_{2}=0$. This yields 
\begin{eqnarray}
\int_{M}\left\vert \nabla ^{2}u\right\vert ^{2}u^{q} &\geq &\frac{m-\left(
m-1\right) k}{m+\left( m-1\right) k}\rho \int_{M}\left\vert \nabla
u\right\vert ^{2}u^{q}  \label{e8} \\
&&-\frac{B_{2}}{m+\left( m-1\right) k}\int_{M}\left\vert \nabla u\right\vert
^{4}u^{q-2},  \notag
\end{eqnarray}%
for $B_{2}$ specified in (\ref{e6}). It is more convenient to denote%
\begin{equation}
2\sigma =\frac{1-k}{k}q,  \label{e9}
\end{equation}%
where we are assuming $\sigma >0$, so in particular $0\leq k\leq \frac{m}{m-1%
}$. Then (\ref{e7}) becomes 
\begin{equation}
a=\frac{3q}{4\left( m+1\right) }-\frac{m}{2\left( m+1\right) }\sigma
\label{e10}
\end{equation}%
and (\ref{e8}) yields 
\begin{eqnarray}
\int_{M}\left\vert \nabla ^{2}u\right\vert ^{2}u^{q} &\geq &\frac{q+2m\sigma 
}{\left( 2m-1\right) q+2m\sigma }\rho \int_{M}\left\vert \nabla u\right\vert
^{2}u^{q}  \label{e11} \\
&&-\frac{1}{\left( 1+\frac{2}{q}\sigma \right) \left( 2m-1+\frac{2m}{q}%
\sigma \right) }B_{3}\int_{M}\left\vert \nabla u\right\vert ^{4}u^{q-2} 
\notag
\end{eqnarray}%
where 
\begin{eqnarray*}
B_{3} &=&\left( \left( 2m-1\right) +\frac{2\left( m-1\right) }{q}\sigma
\right) a^{2}+2\left( \left( 1-q\right) +\left( \frac{2}{q}-1\right) \sigma
\right) ma \\
&&+m\sigma ^{2}+\frac{q-1}{2}\left( q-2\left( m-1\right) \sigma -\frac{4}{q}%
m\sigma ^{2}\right) .
\end{eqnarray*}

Using (\ref{e10}) we obtain 
\begin{eqnarray*}
B_{3} &=&\frac{\left( 2m-1\right) q}{16\left( m+1\right) ^{2}}\left( 8\left(
m+1\right) -\left( 8m-1\right) q\right) \\
&&+\frac{3m-1}{8\left( m+1\right) ^{2}}\left( 8\left( m+1\right) -\left(
8m-1\right) q\right) \sigma \\
&&+\frac{m}{4\left( m+1\right) ^{2}}\left( \left( 2m^{2}-11m+2\right) +\frac{%
8\left( m+1\right) }{q}\right) \sigma ^{2} \\
&&+\frac{m^{2}\left( m-1\right) }{2\left( m+1\right) ^{2}q}\sigma ^{3}.
\end{eqnarray*}%
Plugging (\ref{e11}) into (\ref{e4}) yields 
\begin{eqnarray}
\frac{p-1}{p}\Lambda ^{\prime \prime }\left( t\right) &\geq &\left( \frac{%
q+2m\sigma }{\left( 2m-1\right) q+2m\sigma }+1\right) 2\rho
\int_{M}\left\vert \nabla u\right\vert ^{2}u^{q}  \label{e12} \\
&&+\Upsilon \int_{M}\left\vert \nabla u\right\vert ^{4}u^{q-2},  \notag
\end{eqnarray}%
where 
\begin{equation*}
\Upsilon =q-\frac{2B_{3}}{\left( 1+\frac{2}{q}\sigma \right) \left( 2m-1+%
\frac{2m}{q}\sigma \right) }.
\end{equation*}%
It can be checked directly that $\Upsilon \geq 0$ is equivalent to $\mathcal{%
E}\geq 0,$ where 
\begin{eqnarray*}
\mathcal{E} &=&\left( 2m-1\right) q\left( 8m\left( m+1\right) +\left(
8m-1\right) q\right) \\
&&+2\left( 3m-1\right) \left( 8m\left( m+1\right) +\left( 8m-1\right)
q\right) \sigma \\
&&+4m\left( 8m\left( m+1\right) \frac{1}{q}-\left( 2m^{2}-11m+2\right)
\right) \sigma ^{2} \\
&&-8m^{2}\left( m-1\right) \frac{1}{q}\sigma ^{3}.
\end{eqnarray*}%
It is easy to see that $\mathcal{E}\geq 0$ for $\sigma =1+\frac{q}{2m}$,
which by (\ref{e12}) proves%
\begin{equation*}
\frac{p-1}{p}\Lambda ^{\prime \prime }\left( t\right) \geq \frac{\left(
m+1\right) q+2m}{m\left( q+1\right) }2\rho \int_{M}\left\vert \nabla
u\right\vert ^{2}u^{q}.
\end{equation*}%
Together with (\ref{e3}) this yields 
\begin{eqnarray*}
\Lambda ^{\prime \prime }\left( t\right) &\geq &\frac{\left( m+1\right) q+2m%
}{m\left( q+1\right) }2\rho \left( -\Lambda ^{\prime }\left( t\right) \right)
\\
&=&\frac{\left( m-1\right) p+2}{m}2\rho \left( -\Lambda ^{\prime }\left(
t\right) \right) .
\end{eqnarray*}%
Integrating first from $t=0$ to $t=s$ and then from $s=0$ to $s=\infty $
implies 
\begin{eqnarray*}
\int_{M}\phi ^{2}-\left( \int_{M}\phi ^{\frac{2}{p}}\right) ^{p} &\leq
&\left( -\Lambda ^{\prime }\left( 0\right) \right) \frac{m}{\left(
m-1\right) p+2}\frac{1}{2\rho } \\
&=&\frac{p-1}{p}\frac{m}{\left( m-1\right) p+2}\frac{2}{\rho }%
\int_{M}\left\vert \nabla \phi \right\vert ^{2}.
\end{eqnarray*}%
This proves the result.
\end{proof}

\section{Diameter estimate for K\"{a}hler manifolds\label{Diam}}

In this section, we use the inequalities obtained in the previous sections
to prove Theorem \ref{A}.

\begin{theorem}
\label{A'}Let $\left( M^{m},g\right) $ be a K\"{a}hler manifold of complex
dimension $m\geq 2$ and with Ricci curvature $\mathrm{Ric}\geq \rho $. Then 
\begin{equation*}
\mathrm{diam}\left( M\right) \leq \frac{\pi }{\sqrt{\rho }}\sqrt{2m-1}\left(
1-\frac{1}{24m}\right)
\end{equation*}
\end{theorem}

\begin{proof}
According to \cite{BL}, if the Sobolev inequality
\begin{equation*}
\left( \int_{M}\left\vert \phi \right\vert ^{p}\right) ^{\frac{2}{p}%
}-\int_{M}\phi ^{2}\leq A\int_{M}\left\vert \nabla \phi \right\vert ^{2}
\end{equation*}%
holds for some $p>2,$ then 
\begin{equation*}
\mathrm{diam}\left( M\right) \leq \pi \frac{\sqrt{2pA}}{p-2}.
\end{equation*}%
In our setting, applying Proposition \ref{C} for 
\begin{equation}
p=1+\frac{m+1}{m-1}\frac{4k}{\left( k+1\right) ^{2}},  \label{a19}
\end{equation}%
we get%
\begin{equation}
\mathrm{diam}\left( M\right) \leq \frac{\pi }{\sqrt{\rho }}\sqrt{\frac{%
p\left( m+\left( m-1\right) k\right) }{m\left( p-2\right) }}.  \label{a20}
\end{equation}%
Recall that the Bonnet-Myers estimate is 
\begin{equation*}
\mathrm{diam}\left( M\right) \leq \frac{\pi }{\sqrt{\rho }}\sqrt{2m-1}.
\end{equation*}%
To show that (\ref{a20}) improves this estimate, we compute 
\begin{eqnarray}
&&\sqrt{2m-1}-\sqrt{\frac{p\left( m+\left( m-1\right) k\right) }{m\left(
p-2\right) }}  \label{a20.1} \\
&=&\frac{1}{m\left( p-2\right) }\left( \sqrt{2m-1}+\sqrt{\frac{p\left(
m+\left( m-1\right) k\right) }{m\left( p-2\right) }}\right) ^{-1}\mathcal{S},
\notag
\end{eqnarray}%
where 
\begin{equation*}
\mathcal{S}=m\left( 2m-1\right) \left( p-2\right) -p\left( m+\left(
m-1\right) k\right) .
\end{equation*}%
Using $p$ from (\ref{a19}) we get 
\begin{eqnarray}
\mathcal{S} &=&m\left( 2m-1\right) \left( \frac{m+1}{m-1}\frac{4k}{\left(
k+1\right) ^{2}}-1\right)  \label{a21} \\
&&-\left( m+\left( m-1\right) k\right) \left( \frac{m+1}{m-1}\frac{4k}{%
\left( k+1\right) ^{2}}+1\right)  \notag \\
&=&\frac{1}{\left( m-1\right) \left( k+1\right) ^{2}}\Omega ,  \notag
\end{eqnarray}%
where 
\begin{eqnarray*}
\Omega &=&4m\left( 2m-1\right) \left( m+1\right) k-m\left( 2m-1\right)
\left( m-1\right) \left( k+1\right) ^{2} \\
&&-4\left( m+\left( m-1\right) k\right) \left( m+1\right) k-\left( m+\left(
m-1\right) k\right) \left( m-1\right) \left( k+1\right) ^{2} \\
&=&4\left( m+1\right) \left( m-1\right) \left( 2m-k\right) k-\left(
m-1\right) \left( 2m^{2}+\left( m-1\right) k\right) \left( k+1\right) ^{2}.
\end{eqnarray*}%
Hence, using this in (\ref{a21}) we get 
\begin{equation}
\mathcal{S}=\frac{1}{\left( k+1\right) ^{2}}\Psi ,  \label{a22}
\end{equation}%
where 
\begin{eqnarray*}
\Psi &=&4\left( m+1\right) \left( 2m-k\right) k-\left( 2m^{2}+\left(
m-1\right) k\right) \left( k+1\right) ^{2} \\
&=&\left( 1-k\right) \left( \left( m-1\right) \left( 1-k\right) ^{2}-\left(
2m^{2}+9m-1\right) \left( 1-k\right) +8m\right) .
\end{eqnarray*}%
We now set 
\begin{equation}
k=1-\frac{1}{2m},  \label{a23}
\end{equation}%
for which we note that 
\begin{eqnarray}
p-2 &=&\frac{1}{\left( m-1\right) \left( k+1\right) ^{2}}\left( 8-8\left(
1-k\right) -\left( m-1\right) \left( 1-k\right) ^{2}\right)  \label{a24} \\
&>&0.  \notag
\end{eqnarray}

Furthermore, it is easy to see that $\Psi \geq 2.$ Hence, by (\ref{a22}),
this implies%
\begin{equation}
\mathcal{S}\geq \frac{2}{\left( k+1\right) ^{2}}.  \label{a24.1}
\end{equation}

By (\ref{a20.1}) we get 
\begin{eqnarray}
&&\sqrt{2m-1}-\sqrt{\frac{p\left( m+\left( m-1\right) k\right) }{m\left(
p-2\right) }}  \label{a24.2} \\
&=&\frac{2}{m\left( p-2\right) \left( k+1\right) ^{2}}\left( \sqrt{2m-1}+%
\sqrt{\frac{p\left( m+\left( m-1\right) k\right) }{m\left( p-2\right) }}%
\right) ^{-1}.  \notag
\end{eqnarray}%
From (\ref{a24}) and (\ref{a23}) we have 
\begin{eqnarray*}
m\left( p-2\right) \left( k+1\right) ^{2} &=&\frac{m}{m-1}\left( 8-\frac{4}{m%
}-\frac{m-1}{4m^{2}}\right) \\
&\leq &\frac{8m}{m-1}.
\end{eqnarray*}%
Furthermore, note that 
\begin{equation*}
p\left( m+\left( m-1\right) k\right) \leq \frac{2m}{m-1}\left( 2m-1\right) .
\end{equation*}

Hence, (\ref{a24.2}) yields%
\begin{equation*}
\sqrt{2m-1}-\sqrt{\frac{p\left( m+\left( m-1\right) k\right) }{m\left(
p-2\right) }}\geq \frac{m-1}{8m\sqrt{2m-1}}.
\end{equation*}%
It follows that 
\begin{equation}
\sqrt{2m-1}-\sqrt{\frac{p\left( m+\left( m-1\right) k\right) }{m\left(
p-2\right) }}\geq \frac{\sqrt{2m-1}}{24m},  \label{a25}
\end{equation}%
for all $m\geq 2$.

By (\ref{a20}) and (\ref{a25}) we obtain 
\begin{equation*}
\mathrm{diam}\left( M\right) \leq \frac{\pi }{\sqrt{\rho }}\sqrt{2m-1}\left(
1-\frac{1}{24m}\right) .
\end{equation*}%
This proves the theorem.
\end{proof}

We conclude with a different approach to the diameter estimate, which we learned from  Jiaping Wang.

\begin{proposition}
\label{P}Let $\left( M^{m},g\right) $ be a K\"{a}hler manifold of complex
dimension $m\geq 2$ and with Ricci curvature $\mathrm{Ric}\geq 2m-1$. Let $d=%
\mathrm{diam}\left( M\right) $ be the diameter of $M$. Then%
\begin{equation*}
2\left( 2m-1\right) \leq \left( \frac{\pi }{d}\right) ^{2}\frac{\int_{0}^{%
\frac{d}{2}}\sin ^{2}\left( \frac{\pi r}{d}\right) \sin ^{2m-1}rdr}{%
\int_{0}^{\frac{d}{2}}\cos ^{2}\left( \frac{\pi r}{d}\right) \sin ^{2m-1}rdr}
\end{equation*}
\end{proposition}

\begin{proof}
Recall that on a K\"{a}hler manifold with $\mathrm{Ric}\geq 2m-1$ the first
nonzero eigenvalue of the Laplacian satisfies 
\begin{equation}
\lambda _{1}\left( M\right) \geq 2\left( 2m-1\right) .  \label{p1}
\end{equation}%
Indeed, this also follows by setting $p=2$ in Proposition \ref{D}.

Let $p,q\in M$ be so that $d=d\left( p,q\right) $. Then we have the
following well known estimate 
\begin{equation}
\lambda _{1}\left( M\right) \leq \max \left\{ \mu _{1}\left( B\left( p,\frac{%
d}{2}\right) ,\mu _{1}\left( B\left( q,\frac{d}{2}\right) \right) \right)
\right\} ,  \label{p2}
\end{equation}%
where $\mu _{i}\left( \Omega \right) $ denotes the first Dirichlet
eigenvalue of $\Omega $. By using Cheng's comparison theorem \cite{C}, we
know that 
\begin{equation*}
\max \left\{ \mu _{1}\left( B\left( p,\frac{d}{2}\right) ,\mu _{1}\left(
B\left( q,\frac{d}{2}\right) \right) \right) \right\} \leq \mu _{1}\left( 
\widetilde{B}\left( \frac{d}{2}\right) \right) ,
\end{equation*}%
where $\mu _{1}\left( \widetilde{B}\left( \frac{d}{2}\right) \right) $
denotes the first Dirichlet eigenvalue of the ball of radius $\frac{d}{2}$
in the sphere $\left( \mathbb{S}^{2m},\widetilde{g}\right) $ normalized so
that $\mathrm{Ric}_{\widetilde{g}}=2m-1$. Hence, (\ref{p2}) and (\ref{p1})
imply that 
\begin{equation}
2\left( 2m-1\right) \leq \mu _{1}\left( \widetilde{B}\left( \frac{d}{2}%
\right) \right) .  \label{p3}
\end{equation}%
To estimate the right side of (\ref{p3}) we use the variational
characterization%
\begin{equation}
\bar{\mu}_{1}\left( \widetilde{B}\left( \frac{d}{2}\right) \right) \leq 
\frac{\int_{\widetilde{B}\left( \frac{d}{2}\right) }\left\vert \nabla \phi
\right\vert _{\widetilde{g}}^{2}}{\int_{\widetilde{B}\left( \frac{d}{2}%
\right) }\phi ^{2}},  \label{p4}
\end{equation}%
for any function $\phi $ supported in $\widetilde{B}\left( \frac{d}{2}%
\right) $. Choosing a rotationally symmetric function 
\begin{equation*}
\phi \left( r\right) =\cos \left( \frac{r\pi }{d}\right) ,
\end{equation*}%
we get by (\ref{p4}) that 
\begin{equation*}
\bar{\mu}_{1}\left( \widetilde{B}\left( \frac{d}{2}\right) \right) \leq
\left( \frac{\pi }{d}\right) ^{2}\frac{\int_{0}^{\frac{d}{2}}\sin ^{2}\left( 
\frac{\pi r}{d}\right) \sin ^{2m-1}rdr}{\int_{0}^{\frac{d}{2}}\cos
^{2}\left( \frac{\pi r}{d}\right) \sin ^{2m-1}rdr}.
\end{equation*}%
Using (\ref{p3}), this proves the proposition.
\end{proof}

We use Proposition \ref{P} to prove the following result.

\begin{theorem}
\label{A''}Let $\left( M^{m},g\right) $ be a K\"{a}hler manifold of complex
dimension $m\geq 2$ and with Ricci curvature $\mathrm{Ric}\geq 2m-1$. Then 
\begin{equation*}
\mathrm{diam}\left( M\right) \leq \pi \left(1 -\frac{1}{200\sqrt{m}\ln m}\right).
\end{equation*}
\end{theorem}

\begin{proof}
Let 
\begin{equation}
\varepsilon =\frac{\pi }{d}-1>0  \label{m1}
\end{equation}%
and assume by contradiction that 
\begin{equation}
\varepsilon <\frac{1}{100\sqrt{m}\ln m}.  \label{m2}
\end{equation}%
We have the inequalities%
\begin{eqnarray}
1 &\geq &\cos \left( \varepsilon t\right) \geq 1-\frac{1}{2}\varepsilon
^{2}t^{2}  \label{m3} \\
\varepsilon t &\geq &\sin \left( \varepsilon t\right) \geq \frac{1}{2}%
\varepsilon t,  \notag
\end{eqnarray}%
for any $0<t<\frac{\pi }{2}.$ Denote with%
\begin{equation*}
I_{n}=\int_{0}^{\frac{d}{2}}\sin ^{n}rdr.
\end{equation*}%
Integrating by parts we get the relation 
\begin{equation}
I_{2k+1}=\frac{2k}{2k+1}I_{2k-1}-\frac{1}{2k+1}\sin ^{2k}\left( \frac{d}{2}%
\right) \cos \left( \frac{d}{2}\right) ,  \label{m3'}
\end{equation}%
for any $k\geq 1$. As 
\begin{eqnarray*}
\cos \left( \frac{d}{2}\right) &=&\cos \left( \frac{\pi }{2}-\frac{%
\varepsilon }{1+\varepsilon }\frac{\pi }{2}\right) \\
&=&\sin \left( \frac{\varepsilon }{1+\varepsilon }\frac{\pi }{2}\right) ,
\end{eqnarray*}%
we get from (\ref{m2}) and (\ref{m3}) that%
\begin{equation}
\frac{1}{2}\varepsilon \leq \cos \left( \frac{d}{2}\right) \leq 2\varepsilon
.  \label{m4}
\end{equation}%
This implies that 
\begin{equation}
I_{2k+1}\geq \frac{2k}{2k+1}I_{2k-1}-\frac{2\varepsilon }{2k+1}.  \label{m5}
\end{equation}%
Iterating, we get%
\begin{equation*}
I_{2m+1}\geq \frac{2^{2m}\left( m!\right) ^{2}}{\left( 2m+1\right) !}\left(
1-\cos \frac{d}{2}\right) -2\varepsilon \left( \frac{1}{2m+1}+\frac{1}{2m-1}%
+...+\frac{1}{3}\right) .
\end{equation*}%
Using Stirling inequalities 
\begin{equation*}
\sqrt{2\pi }\left( \frac{n}{e}\right) ^{n}\sqrt{n}\leq n!\leq e\left( \frac{n%
}{e}\right) ^{n}\sqrt{n},
\end{equation*}%
together with (\ref{m2}) and (\ref{m4}) it follows that 
\begin{equation}
I_{2m+1}\geq \frac{1}{\sqrt{2m+1}}-2\varepsilon \ln \left( 2m+1\right) .
\label{m6}
\end{equation}%
In particular, (\ref{m2}) and (\ref{m6}) imply that 
\begin{equation}
I_{2m+1}\geq \frac{4}{5}\frac{1}{\sqrt{2m+1}}.  \label{m6'}
\end{equation}

Note that 
\begin{equation}
\frac{\int_{0}^{\frac{d}{2}}\sin ^{2}\left( \frac{\pi r}{d}\right) \sin
^{2m-1}rdr}{\int_{0}^{\frac{d}{2}}\cos ^{2}\left( \frac{\pi r}{d}\right)
\sin ^{2m-1}rdr}=\frac{I_{2m-1}}{\int_{0}^{\frac{d}{2}}\cos ^{2}\left( \frac{%
\pi r}{d}\right) \sin ^{2m-1}rdr}-1.  \label{m7}
\end{equation}%
From (\ref{m3}) we get 
\begin{eqnarray*}
\cos ^{2}\left( \frac{\pi r}{d}\right) &=&\left( \cos \left( \varepsilon
r\right) \cos r-\sin \left( \varepsilon r\right) \sin r\right) ^{2} \\
&\geq &\cos ^{2}\left( \varepsilon r\right) \cos ^{2}r-2\sin \left(
\varepsilon r\right) \cos \left( \varepsilon r\right) \sin r\cos r \\
&\geq &\left( 1-\varepsilon ^{2}r^{2}\right) \cos ^{2}r-2\varepsilon r\sin
r\cos r \\
&\geq &\cos ^{2}r-4\varepsilon \cos r.
\end{eqnarray*}%
Then it follows that 
\begin{eqnarray*}
\int_{0}^{\frac{d}{2}}\cos ^{2}\left( \frac{\pi r}{d}\right) \sin ^{2m-1}rdr
&\geq &\int_{0}^{\frac{d}{2}}\cos ^{2}r\sin ^{2m-1}rdr-4\varepsilon
\int_{0}^{\frac{d}{2}}\cos r\sin ^{2m-1}rdr \\
&=&I_{2m-1}-I_{2m+1}-\frac{2\varepsilon }{m}\sin ^{2m}\left( \frac{d}{2}%
\right) \\
&\geq &I_{2m-1}-I_{2m+1}-\frac{2\varepsilon }{m}.
\end{eqnarray*}%
From (\ref{m3'}), we deduce that 
\begin{equation*}
\int_{0}^{\frac{d}{2}}\cos ^{2}\left( \frac{\pi r}{d}\right) \sin
^{2m-1}rdr\geq \frac{1}{2m+1}I_{2m-1}-\frac{2\varepsilon }{m}.
\end{equation*}%
However, by \ (\ref{m2}) and (\ref{m6'}) we get%
\begin{equation*}
\frac{2\varepsilon }{m}<\frac{1}{3}\frac{1}{2m+1}I_{2m+1}.
\end{equation*}%
Therefore, this implies 
\begin{equation}
\int_{0}^{\frac{d}{2}}\cos ^{2}\left( \frac{\pi r}{d}\right) \sin
^{2m-1}rdr\geq \frac{2}{3}\frac{1}{2m+1}I_{2m-1}.  \label{m8}
\end{equation}%
Hence, by (\ref{m7}) and (\ref{m8}) we infer that 
\begin{equation*}
\frac{\int_{0}^{\frac{d}{2}}\sin ^{2}\left( \frac{\pi r}{d}\right) \sin
^{2m-1}rdr}{\int_{0}^{\frac{d}{2}}\cos ^{2}\left( \frac{\pi r}{d}\right)
\sin ^{2m-1}rdr}\leq \frac{3}{2}\left( 2m+1\right) .
\end{equation*}%
By Proposition \ref{P} we get 
\begin{equation*}
2\left( 2m-1\right) \leq \frac{3}{2}\left( 1+\varepsilon \right) ^{2}\left(
2m+1\right) ,
\end{equation*}%
which contradicts (\ref{m2}).

Therefore, we have%
\begin{equation*}
\frac{\pi }{d}-1\geq \frac{1}{100\sqrt{m}\ln m},
\end{equation*}%
which proves the result.
\end{proof}

\end{document}